\documentclass[12pt]{amsart}
\usepackage[utf8]{inputenc}
\usepackage{amsmath,amssymb,amsthm}
\usepackage{tikz,tikz-cd}
\usepackage{lipsum} 

\DeclareMathOperator{\Pic}{\mathrm{Pic}}

\newtheorem{conjecture}[equation]{Conjecture} 
 
\newtheorem{theorem}[equation]{Theorem}

\newtheorem{corollary}[equation]{Corollary}
\theoremstyle{definition}
\newtheorem{definition}{Definition}[section]

\theoremstyle{remark}
\newtheorem{remark}[equation]{Remark}
\newtheorem{notation}[equation]{Notation}

\begin{document}
\begin{abstract} We use the canonical bundle formula for parabolic fibrations to give an inductive approach to the generalized abundance conjecture using nef reduction. In particular, we observe that generalized abundance holds for a klt pair $(X,B)$ if the nef dimension $n(K_X+B+L)=2$ and $K_X+B \geq 0$ or $n(K_X+B+L)=3$ and $\kappa(K_X+B )>0$.

\end{abstract}

\title[Nef reduction]{An inductive approach to generalized abundance using nef reduction}
\author{Priyankur Chaudhuri}
\address{Department of Mathematics, University of Maryland, College Park, MD 20742}
\email{pchaudhu@umd.edu}

\maketitle
\section{Introduction} In \cite{LP}, Lazić and Peternell propose the following

\begin{conjecture}(Generalized Abundance) Let $(X,B)$ be a klt pair with $K_X+B$ pseudoeffective. If $L$ is a nef divisor on $X$ such that $K_X+B+L $ is also nef, then $K_X+B+L \equiv M$ for some semiample $\mathbb{Q}$-divisor $M$.\end{conjecture}

They show that this conjecture holds in case $\dim X =2$ and for $\dim X=3$ if $\kappa(K_X+B) >0$ (see \cite[Corollary C,D]{LP}). The main purpose of this article is to prove the following result.
\begin{theorem} \label{1}
Let $(X,B)$ be an $n$-dimensional klt pair with $K_X+B \geq 0$ and let $L \in \Pic X$ be nef such that $K_X+B+L$ is nef with nef dimension $n(K_X+B+L) = d$. 
\newline
Assume termination of klt flips \cite[Conjecture 5-1-13]{KMM} in dimension $d$, abundance conjecture \cite[Conjecture 3.12]{KM} in dimension $\leq d$, semiampleness conjecture \cite[page 2]{LP} in dimension $\leq d$ and generalized non-vanishing conjecture \cite[page 2]{LP} in dimension $ \leq d-1$. 
\newline
Then $K_X+B+L \equiv M$ for some semiample divisor $M$.
\end{theorem}

Note that the phrase ``Conjecture ($*$) holds in dimension $d$ (resp. $\leq d$)" above means that ($*$) holds for all klt pairs $(X,B)$ such that $\dim X =d$ (resp. $\leq d$).\\

In case $L=0$, such a theorem has been proved by \cite{Am}. The main ingredients of our proof are the canonical bundle formula for parabolic fibrations \cite{FM} and Nakayama-Zariski decomposition \cite{Nak}. We have the following corollary

\begin{corollary} \label{10}
Generalized abundance holds (in any dimension) in the following two cases:

\begin{enumerate}
\item $n(K_X+B+L)=2$ and $K_X+B \geq 0$, or 
\item $n(K_X+B+L)=3$ and $\kappa(K_X+B) >0$.
\end{enumerate}
\end{corollary}

Note that any divisor of nef dimension $1$ is already numerically equivalent to a semiample divisor by \cite[2.4.4]{8 authors}. Thus generalized abundance also holds in this case.

\section{Preliminaries}
\begin{definition}
[Singularities of pairs]  Let $(X,B)$ be a sub-pair consisting of a normal variety $X$ and a $\mathbb{Q}$-divisor $B$ on $X$ such that $K_X+B$ is $\mathbb{Q}$-Cartier. It is called \emph{sub-klt} if there exists a log resolution $Y \xrightarrow {\mu} X$ of $(X,B)$ such that letting $B_Y$ be defined by $K_Y+B_Y= \mu^*(K_X+B)$, all coefficients of $B_Y$ are $<1$ and \emph{sub-lc} if all coefficients of $B_Y$ are $\leq 1$. If $B \geq 0$, we drop the prefix sub.
\end{definition}

\begin{definition}
[Nef reduction and nef dimension \cite{8 authors}] Let $X$ be a normal projective variety and let $L \in \Pic X$ be nef. Then there exists a dominant rational map $ \phi: X \dashrightarrow  Y$ with connected fibers which is proper and regular over an open subset of $Y$ (i.e. there exists $V \subset Y$ nonempty open such $\phi$ restricts to a proper morphism $ \phi|_{\phi^{-1}(V)}:\phi^{-1}(V) \rightarrow V$) where $Y$ is also normal projective such that

\begin{enumerate}
\item If $F \subset X$ is a general compact fiber of $\phi $ with $\dim F= \dim X - \dim Y$, then $L|_{F} \equiv 0 $.

\item If $x \in X$ is a very general point and $C \subset X$ a curve passing $x$ such that $\dim (\phi(C)) >0 $, then $(L \cdot C) > 0$.

\end{enumerate}

$\phi$ is called the \textbf{nef reduction map} of $L$ and $\dim Y$ the \textbf{nef dimension} $n(L)$ of $L$. 

\end{definition}

\begin{remark} Note that if $\phi: X \dashrightarrow Y$ is regular over an open subset of $Y$, then there exists a resolution $\hat{\phi}:\hat{X} \rightarrow Y$ of $\phi$ such that the exceptional divisor $Ex (\hat{X}/X)$ is $\hat{\phi}$-vertical. For example, $\hat{X}$ can be chosen to be the closure of the graph of $\phi$.
\end{remark}

\begin{notation}
For a pseudoeffective divisor $D$ on a smooth projective variety $X$, $P_{\sigma}(D)$ and $N_{\sigma}(D)$ will denote the positive and negative parts of the Nakayama-Zariski decomposition of $D$ (see \cite[Chapter 3]{Nak} for details). $\kappa(D)$ will denote the Iitaka dimension and $\nu(D)$ the numerical dimension (see \cite[Definition 2.10]{GL}) of $D$.

\end{notation}

\section{Inductive approach to generalized abundance}

\begin{proof}[Proof of Theorem \ref{1}]  We will follow some ideas of \cite[Theorem 5.1]{Am} and \cite[Lemma 4.4]{GL}. Let $\Phi$ be the nef reduction map of $K_X+B+L$. Let $\hat{X}$ be the normalization of the closure of the graph of $\Phi $. We have an induced commutative diagram:

\begin{center}
\begin{tikzcd} 
\hat{X} \arrow[d,"\beta"] \arrow[dr, "\phi"] \\
X \arrow[r, dotted, "\Phi"] & Y

\end{tikzcd}
\end{center}

Note that since $\Phi$ is a regular over an open subset of $Y$, $Ex(\beta)$ is $\phi$-vertical. We will make base changes that preserve this property. Let $F \subset \hat{X}$ be a general fiber of $\phi $. Then 
\begin{equation}
(K_{\hat{X}}+B_{\hat{X}}+L_{\hat{X}})|_{F} \equiv 0 \label{triv}.
\end{equation}
Here $L_{\hat{X}}:= \beta^*L$ and $B_{\hat{X}}$ is defined by $K_{\hat{X}}+B_{\hat{X}} = \beta^*(K_X+B)$. Now since $K_{\hat{X}}+B_{\hat{X}} \geq 0$, we have $K_F+B_F \geq 0$ which implies $K_F+B_F \sim_{\mathbb{Q}} 0$. Indeed, suppose that $K_F+B_F >0$. Then there exists a curve $C \subset F$ such that $((K_F+B_F)\cdot C) >0$. Then $((K_F+B_F+L_F) \cdot C) >0$. This contradicts (\ref{triv}). We also conclude that $L_{\hat{X}}|_F \equiv 0$.\\

By \cite[Lemma 3.1]{LP}, there exists $\sigma: Y^{'} \rightarrow Y$ birational from a smooth projective variety $Y^{'}$ such that letting $X^{'}$ denote the normalization of the main component of $\hat {X} \times _Y Y^{'} \rightarrow Y^{'}$ and $\phi^{'}: X^{'} \rightarrow Y^{'}$ the induced morphism, there exists a nef $\mathbb{Q}$-Cartier divisor $L_{Y^{'}}$ on $Y^{'}$ such that $L_{X^{'}} \equiv \phi^{'*}(L_{Y^{'}})$. Letting $F^{'}$ denote a general fiber of $\phi^{'}$, note that $(B_{X^{'}}^{-})|_{F^{'}} \sim 0$. Thus 

\begin{center}
$(K_{X^{'}}+B_{X^{'}})|_{F^{'}} \sim (K_{X^{'}}+B_{X^{'}}^{+})|_{F^{'}} \sim 0$.
\end{center}

Then by \cite[Section 4]{FM}, there exists a birational morphism $\widetilde{Y} \rightarrow Y^{'}$, $\widetilde{X}$ birational to the main component of $X^{'}\times _{Y^{'}} \widetilde{Y}$ where $\widetilde{X}$ and $\widetilde{Y}$ are both smooth projective such that letting $\widetilde{\phi}: \widetilde{X} \rightarrow \widetilde{Y}$ denote the induced morphism, there exists a $\mathbb{Q}$-divisor $\Delta$ on $\widetilde{X}$ such that $\widetilde{\phi}: (\widetilde{X}, B_{\widetilde{X}}^{+}+ \Delta) \rightarrow \widetilde{Y}$ is a klt trivial fibration (\cite[Definition 2.1]{LF}) where $\Delta^{+}$ is exceptional over $\widetilde{Y}$ and $X^{'}$ and $\widetilde{\phi}_{*} \mathcal{O}_{\widetilde{X}}(\left\lfloor{l \Delta^{-}}\rfloor\right) \cong \mathcal{O}_{\widetilde{Y}}$ for all $l \in \mathbb{N}$.

\begin{center}
\begin{tikzcd}
\widetilde{X} \arrow[r,"\pi"] \arrow[d, "\widetilde{\phi}"] & X^{'} \arrow[r] \arrow[d, "\phi ^{'}"] &\hat{X} \arrow[r, "\beta"] \arrow[d,"\phi"] & X \arrow[dl, dotted]\\
\widetilde{Y} \arrow[r] &Y^{'} \arrow[r] &Y
\end{tikzcd}
\end{center}
\vspace{0.3 cm}

 By \cite[Lemma 2.15]{GL}, the last condition implies that the vertical part $(\Delta^{-})^{v}$ is $\widetilde{\phi}$-degenerate (see \cite[Definition 2.14]{GL}). Letting $B_{\widetilde{Y}}$ and $M_{\widetilde{Y}}$ denote the induced discriminant and moduli divisors on $\widetilde{Y}$, we have

\begin{center}
$K_{\widetilde{X}}+B_{\widetilde{X}}^{+}+\Delta^{+}-\Delta^{-}\sim_{\mathbb{Q}} \widetilde{\phi}^{*}(K_{\widetilde{Y}}+B_{\widetilde{Y}}+M_{\widetilde{Y}})$.
\end{center}

Since $(\Delta^{-})^{v}$ is $ \widetilde{\phi}$-degenerate, it follows from the definition of $B_{\widetilde{Y}}$ that $B_{\widetilde{Y}} \geq 0$. Since $K_{\widetilde{X}}+B_{\widetilde{X}}^{+}=\pi^*(K_{X^{'}}+B_{X^{'}}^+)+E$, where $E \geq 0$ is $\pi$-exceptional, if $\widetilde{F} \subset \widetilde{X}$ is a general fiber of $\widetilde{\phi}$, $(K_{\widetilde{X}}+B_{\widetilde{X}}^+)|_{\widetilde{F}} =E|_{\widetilde{F}}$. Note that $E|_{\widetilde{F}}$ is an exceptional divisor for the induced birational morphism $\pi|_{\widetilde{F}}: \widetilde{F} \rightarrow F^{'}$. So $\kappa((K_{\widetilde{X}}+B_{\widetilde{X}}^+)|_{\widetilde{F}})= \nu((K_{\widetilde{X}}+B_{\widetilde{X}}^+)|_{\widetilde{F}})=0$. Thus we can run a relative MMP with ample scaling over $\widetilde{Y}$ to get 

\begin{center}
\begin{tikzcd}
\widetilde{X} \arrow[r, dotted,"\psi"] \arrow[d, "\widetilde{\phi}"] &\widetilde{X}_m \arrow[dl, "\widetilde{\phi}_m"] \\
\widetilde{Y}
\end{tikzcd}
\end{center}

such that letting $K_{\widetilde{X}_m}+B_{\widetilde{X}_m}^+=\psi_*(K_{\widetilde{X}}+B_{\widetilde{X}}^+)$ and $\Delta_{\widetilde{X}_m}:=\psi_*\Delta$, we have $( K_{\widetilde{X}_m}+B_{\widetilde{X}_m}^+)|_{\widetilde{F}_m} \sim_{\mathbb{Q}} 0$, $\widetilde{F}_m$ being the general fiber of $\widetilde{\phi}_m$ (see the first paragraph of the proof of \cite[Lemma 4.4]{GL} for details). Consider the induced klt-trivial fibration 
\begin{center}
$\widetilde{\phi}_m: (\widetilde{X}_m, B_{\widetilde{X}_m}+\Delta_{\widetilde{X}_m}^+- \Delta_{\widetilde{X}_m}^-) \rightarrow \widetilde{Y}$.
\end{center}
Note: the fact that $\widetilde{\phi}_m$ is klt-trivial can be shown by choosing a smooth resolution of indeterminacy $\widetilde{X} \xleftarrow{p} W \xrightarrow{q} \widetilde{X}_m$ of $\psi$ and using the fact that $p^*(K_{\widetilde{X}}+B_{\widetilde{X}}^+)=q^*(K_{\widetilde{X}_m}+B_{\widetilde{X}_m}^+)+E$ where $E \geq 0$ is $q$-exceptional. By \cite[Lemma 2.6]{Am1}, $ \widetilde{\phi}_m$ and $\widetilde{\phi}$ induce the same discriminant and moduli divisors on $\widetilde{Y}$. In the case of $\widetilde{\phi}_m$, $(\Delta_{\widetilde{X}_m}^-)|_{\widetilde{F}_m} \sim_{\mathbb{Q}} 0$, thus by \cite[Theorem 3.3]{Am2}, $M_{\widetilde{Y}}$ is b-nef and b-good. Hence, by the arguments in the proof of \cite[Theorem 4.1]{Am2}, we can write

\begin{center}
$K_{\widetilde{X}}+B_{\widetilde{X}}^+ + \Delta^+ - \Delta^- \sim_{\mathbb{Q}}\widetilde{\phi}^*(K_{\widetilde{Y}}+\Delta_{\widetilde{Y}})$.
\end{center} 
where $(\widetilde{Y}, \Delta_{\widetilde{Y}})$ is klt. We write this as 
\begin{equation}
K_{\widetilde{X}}+B_{\widetilde{X}}^+ + \Delta^+ - (\Delta^-)^h \sim_{\mathbb{Q}} \widetilde{\phi}^*(K_{\widetilde{Y}}+\Delta_{\widetilde{Y}})+ (\Delta^-)^v \label{star}
\end{equation}
where $^h$ and $^v$ denote the horizontal and vertical parts with respect to $ \widetilde{\phi} $ respectively. As observed above, $ (K_{\widetilde{X}}+B_{\widetilde{X}}^+)|_{\widetilde{F}} $ is $\pi|_{\widetilde{F}}$-exceptional and $(K_{\widetilde{X}}+B_{\widetilde{X}}^+)|_{\widetilde{F}} \sim_{\mathbb{Q}} (\Delta^-)^h|_{\widetilde{F}}$. Thus $(\Delta^-)^h|_{\widetilde{F}}= N_{\sigma}((K_{\widetilde{X}}+B_{\widetilde{X}}^+)|_{\widetilde{F}})(=(K_{\widetilde{X}}+B_{\widetilde{X}}^+)|_{\widetilde{F}})$. Now it follows from the definition of $N_{\sigma}$ that 
\begin{equation}
N_{\sigma}((K_{\widetilde{X}}+B_{\widetilde{X}}^+)|_{\widetilde{F}}) \leq N_{\sigma}(K_{\widetilde{X}}+B_{\widetilde{X}}^+)|_{\widetilde{F}}.
\end{equation}
Since $\widetilde{F}$ is a general fiber and $(\Delta^-)^h$ has no vertical components, it follows that 
\begin{equation}
(\Delta^-)^h \leq N_{\sigma}(K_{\widetilde{X}}+B_{\widetilde{X}}^+) \label{ddag}
\end{equation}
Since $P_{\sigma}(K_{\widetilde{X}}+B_{\widetilde{X}}^+)$ is effective, it follows that $K_{\widetilde{X}}+B_{\widetilde{X}}^+- (\Delta^-)^h \geq 0$. By \cite[Lemma 1.8]{Nak}, (\ref{ddag}) implies that $N_{\sigma}(K_{\widetilde{X}}+B_{\widetilde{X}}^+- (\Delta)^h)= N_{\sigma}(K_{\widetilde{X}}+B_{\widetilde{X}}^+)-(\Delta^-)^h$ and thus $P_{\sigma}(K_{\widetilde{X}}+B_{\widetilde{X}}^+-(\Delta^-)^h)= P_{\sigma} (K_{\widetilde{X}}+B_{\widetilde{X}}^+)$. Now we take $P_{\sigma}$ in (\ref{star}). Since $ \kappa(D)= \kappa(P_{\sigma}(D))$ for any pseudoeffective divisor $D$ (\cite[Lemma 2.9]{GL}), the fact that $\Delta^+$ and $B_{\widetilde{X}}^-$ are exceptional over $X$, the $\widetilde{\phi}$-degeneracy of $(\Delta^-)^v$ and \cite[Lemma 2.16]{GL} imply that

\begin{equation}
\kappa(K_X+B)= \kappa(K_{\widetilde{Y}}+\Delta_{\widetilde{Y}})  \label{bstar}
\end{equation}
Let $\mu: \widetilde{X} \rightarrow X$ be the induced birational morphism. Then (\ref{star}) gives 
\begin{equation}
K_{\widetilde{X}}+B_{\widetilde{X}}^++L_{\widetilde{X}}+\Delta^+-(\Delta^-)^h \equiv \widetilde{\phi}^*(K_{\widetilde{Y}}+\Delta_{\widetilde{Y}}+L_{\widetilde{Y}})+(\Delta^-)^v \label{dag}
\end{equation}
where $\Delta^+$ is exceptional over $X^{'}$ and $\widetilde{Y}$ and $(\Delta^-)^v$ is $\widetilde{\phi}$-degenerate. Since $L_{\widetilde{X}}|_{\widetilde{F}} \equiv 0$, as observed above, we have  
\begin{center}
$N_{\sigma}((K_{\widetilde{X}}+B_{\widetilde{X}}^++L_{\widetilde{X}})|_{\widetilde{F}})= N_{\sigma}((K_{\widetilde{X}}+B_{\widetilde{X}}^+)|_{\widetilde{F}}) \geq (\Delta^-)^h|_{\widetilde{F}}$ implying $ (\Delta^-)^h \leq N_{\sigma}((K_{\widetilde{X}}+B_{\widetilde{X}}^++L_{\widetilde{X}})$ 
\end{center}
and exactly as before, we conclude that 
\begin{center}
$P_{\sigma}(K_{\widetilde{X}}+B_{\widetilde{X}}^++L_{\widetilde{X}}+\Delta^+-(\Delta^-)^h)= P_{\sigma}(K_{\widetilde{X}}+B_{\widetilde{X}}^++L_{\widetilde{X}}+ \Delta^+)$.
\end{center}
But the latter equals 
\begin{center}
$P_{\sigma}(\mu^*(K_X+B+L)+B_{\widetilde{X}}^-+\Delta^+)=P_{\sigma}(\mu^*(K_X+B+L))=\mu^*(K_X+B+L)$
\end{center}
since $\Delta^+$ and $B_{\widetilde{X}}^-$ are exceptional over $X$ and $K_X+B+L$ is nef. Now taking $P_{\sigma}$ in (\ref{dag}) and using \cite[Lemma 2.16]{GL}, we get
\begin{center}
$\mu^*(K_X+B+L)=P_{\sigma}(\widetilde{\phi}^*(K_{\widetilde{Y}}+\Delta_{\widetilde{Y}}+L_{\widetilde{Y}}))$.
\end{center}
By \cite[Theorem 5.3]{LP}, there exists a smooth projective $\overline{Y}$ and $w: \overline{Y} \rightarrow \widetilde{Y}$ birational such that $P_{\sigma}(w^*(K_{\widetilde{Y}}+\Delta_{\widetilde{Y}}+L_{\widetilde{Y}}))$ is numerically equivalent to a semiample divisor. Letting $\overline{X}$ denote a desingularization of the main component of $\widetilde{X} \times _{\widetilde{Y}}\overline{Y}$ and $v: \overline{X} \rightarrow \widetilde{X}$, $ \overline{\phi}:\overline{X} \rightarrow  \overline{Y}$ the induced morphisms, we have

\begin{center}
$ \overline{\phi}^*P_{\sigma}(w^*(K_{\widetilde{Y}}+\Delta_{\widetilde{Y}}+L_{\widetilde{Y}}))= P_{\sigma}(\overline{\phi}^*w^*(K_{\widetilde{Y}}+\Delta_{\widetilde{Y}}+L_{\widetilde{Y}}))$ (since $P_{\sigma}(w^*(K_{\widetilde{Y}}+\Delta_{\widetilde{Y}}+L_{\overline{Y}}))$ is nef; see \cite[Lemma 2.5]{LP}) $= P_{\sigma}(v^*\widetilde{\phi}^*(K_{\widetilde{Y}}+\Delta_{\widetilde{Y}}+L_{\widetilde{Y}}))= v^*P_{\sigma}(\widetilde{\phi}^*(K_{\widetilde{Y}}+\Delta_{\widetilde{Y}}+L_{\widetilde{Y}}))$ (since $P_{\sigma}(\widetilde{\phi}^*(K_{\widetilde{Y}}+\Delta_{\widetilde{Y}}+L_{\widetilde{Y}}))$ is nef) $=v^* \mu^*(K_X+B+L)$ ($*$)
\end{center}
is numerically equivalent to a semiample divisor. Thus $K_X+B+L$ is numerically equivalent to a semiample divisor.

\end{proof}
\begin{remark}
The above result can also be proved by using the MMP instead of the canonical bundle formula. See for example the proof of \cite[Theorem 3.5]{LP2}.
\end{remark}

\begin{proof}
[Proof of Corollary \ref{10}]
\begin{enumerate}\item If $n(K_X+B+L)=2$, then $\dim \widetilde{Y} =2$ and $K_{\widetilde{Y}}+\Delta_{\widetilde{Y}} \geq 0$ by (\ref{bstar}). Termination of flips for threefold pairs and abundance for surface and threefold pairs are classical. Semiampleness and generalized non-vanishing conjectures hold for surface pairs by \cite[Corollary C, page 4]{LP}. Thus the case $n(K_X+B+L)=2$ i.e. when $\dim \widetilde{Y} =2$ is immediate. \vspace{0.5 cm}
\item
 If $n(K_X+B+L)=3$, then $\dim \widetilde{Y} =3$ and by (\ref{bstar}), $ \kappa(K_{\widetilde{Y}}+\Delta_{\widetilde{Y}}) >0$. Then by \cite[Remark 5.4]{LP}, $P_{\sigma}(w^*(K_{\widetilde{Y}}+\Delta_{\widetilde{Y}}+L_{\widetilde{Y}}))$ is numerically equivalent to a semiample divisor for some birational morphism $w:\overline{Y}\rightarrow \widetilde{Y}$ from a smooth projective variety $\overline{Y}$. Then so is $K_X+B+L$ by ($*$) above.
\end{enumerate}

\end{proof}
\section{\emph{Acknowledgements}} I thank my advisor Patrick Brosnan for several helpful conversations on the contents of this note, Vladimir Lazić for his comments on a draft version, Brian Lehmann for answering a question and the referee whose comments improved the readability of this paper. An email conversation with Florin Ambro provided the impetus for this work.

\end{document}